\documentclass{amsart}

\usepackage{mathtools}
\usepackage{fullpage} 

\usepackage{version}
\excludeversion{diary}

\usepackage{hyperref}

\newtheorem{theorem}{Theorem}[section]
\newtheorem{lemma}[theorem]{Lemma}
\newtheorem{corollary}[theorem]{Corollary}
\newtheorem{proposition}[theorem]{Proposition}

\theoremstyle{definition}

\newtheorem{example}[theorem]{Example}

\theoremstyle{remark}
\newtheorem{remark}[theorem]{Remark}
\newtheorem{conjecture}[theorem]{Conjecture}

\numberwithin{equation}{section}

\DeclarePairedDelimiter\set{\{}{\}}
\DeclarePairedDelimiter\floor{\lfloor}{\rfloor}
\DeclarePairedDelimiter\prnths{(}{)}
\DeclarePairedDelimiter\abs{\lvert}{\rvert}
\DeclarePairedDelimiter\norm{\lVert}{\rVert}
\DeclarePairedDelimiter\cyclic{\langle}{\rangle}
\DeclareMathOperator{\sign}{sign}

\newcommand*{\R}{\mathbb{R}}
\newcommand*{\Z}{\mathbb{Z}}

\newcommand*\D[1]{\Delta_{#1}}
\newcommand*{\PluckerEqTuples}[6]{\D{#1}\D{#2} = \D{#3}\D{#4} + \D{#5}\D{#6}}
\newcommand*{\PluckerEq}[4]{\PluckerEqTuples{#1,#3}{#2,#4}{#1,#2}{#3,#4}{#1,#4}{#3,#2}}
\newcommand*{\Gr}{Gr(2,n)}
\newcommand*{\Grp}{Gr^{>0}(2,n)}
\newcommand*{\GrpK}[1]{Gr^{>0}(#1)}

\begin{document}

\title{The Regular Polygon Minimizes The Ratio of Plucker Coordinates on The Positive Grassmannian}


\author{Vadim Ogranovich}
\address{}
\curraddr{}
\email{vograno@gmail.com}
\thanks{I am grateful to Prof. Alexandre Eremenko for his generous help at various stages of this long project.}


\date{}

\dedicatory{}

\begin{abstract}
For a point $x$ on the Positive Grassmannian of two-dimensional subspaces in $\mathbb{R}^n$, define the loss function $E(x)$ as the ratio of its largest and smallest Plucker coordinates. We solve the extremal problem of minimizing the loss function $E(x)$ over the Grassmannian. This minimax problem was posed by Berman, et al. in their paper on error-correcting codes over the real numbers.
\end{abstract}

\maketitle
\section{Introduction}

In their paper \cite{berman} on optimal error-correcting codes over $\R$, Berman, et al. asked the following question.
Let $X$ be a $2 \times n$ real matrix, and let $\D{i,j}(X)$ denote the minors of $X$,
where $(i,j) \in \binom{[n]}{2}$, and $[n]$ denotes the range of integers from one to $n$,
and $\binom{[n]}{k}$ denotes the set of \emph{ordered} $k$-tuples drawn from the integer range $[n]$
\footnote{We emphasize we use \emph{ordered} pairs $(i,j)$ as the index set of the minors.}.
We often omit the dependency on $X$ and write $\D{i,j}$ when the matrix is clear from the context.
Define the loss function
\begin{equation}
\label{D:Ex}
    E(X) = \frac{\max\limits_{i<j} \D{i,j}}{\min\limits_{i<j} \D{i,j}}.
\end{equation}
The question: what matrix minimizes $E(X)$, subject to the constraints $\D{i,j}>0$, where $(i,j) \in \binom{[n]}{2}$?

\newcommand{\cosk}[1]{\cos #1 \frac{\pi}{n}}
\newcommand{\sink}[1]{\sin #1 \frac{\pi}{n}}
Define the \emph{cyclic matrix} $C$:
\begin{equation}
\label{D:cyclic_matrix}
C =
\begin{bmatrix}
    \cosk{0}, &\cosk{1}, &\dots, &, \cosk{(n-1)} \\
    \sink{0}, &\sink{1}, &\dots, &, \sink{(n-1)}
\end{bmatrix}.
\end{equation}
Our main result is that $C$ minimizes $E(X)$ for any $n$.

The proper framework for the Berman's question is that of the \emph{Grassmannian}, 
the variety $\Gr$ of two-dimensional planes in $\R^n$.
Let $x \in \Gr$, and let $X$ be a $2 \times n$ matrix whose rows span the subspace $x$.
We call $X$ the \emph{spanning} matrix of $x$ and denote its span $[X]$, i.e. $x=[X]$.
The minors $\D{i,j}$ of the matrix $X$ are called the \emph{Plucker coordinates} of $x$.
The \emph{Positive Grassmannian} is the subset $\Grp \subset \Gr$ such that, for $x \in \Grp$, the coordinates $\D{i,j}(x)$ are positive for all $(i,j) \in \binom{[n]}{2}$.

It is well known that, up to a non-zero scaling factor, the Plucker coordinates of a point $x$ on the Grassmannian are independent of the choice of the spanning matrix $X$ (\cite{miller_sturmfels} Proposition 14.2), and therefor the function $E(X)$ lifts to $\Grp$. Thus the extremal problem
\begin{equation}
\label{extremal_grplus}
    E^{*} = \min_{x \in \Grp} E(x)
\end{equation}
is well-defined and is going to be the primary subject of this paper.

An immediate benefit of working on the Grassmannian is the ability to meaningfully address the question of uniqueness. We shall prove that the solution of the extremal problem \eqref{extremal_grplus} is unique for the odd $n$, \emph{not} unique for $n \mod 4 = 2$. Uniqueness remains an open question for $n \mod 4 = 0$, other than $n=4$; for the latter we prove the subspace $[C]$ is the unique extremal point.

\begin{remark}
Our proof is based on the Plucker relations \eqref{D:plucker_identity} and thus is purely algebraic.
There is, however, an insight to be gained from viewing $\Gr$ as a set of $2$-dimensional subspaces in $\R^n$.
Since we do not refer to this remark from elsewhere in the paper, we do not give proofs. We refer the interested reader to \cite{karp} for proofs and a much more general treatment.

Define the linear operator $S$ on $\R^n$:
\begin{equation}
    S: (x_1, \dots, x_n) \rightarrow (-x_n, x_1, \dots, x_{n-1}).
\end{equation}
Let $x \in \Gr$, and let $X = [x_1, \dots, x_n]$ be its spanning matrix, i.e. $x = [X]$.
The spanning matrix of the subspace $S x$ is obtained by applying $S$ to each of the two rows of matrix $X$:
\begin{equation}
    S: [x_1, \dots, x_n] \rightarrow [-x_n, x_1, \dots, x_{n-1}].
\end{equation}
This implies that the Plucker coordinates of the subspace $S x$ are exactly those of the subspace $x$, only cyclically shifted to the right. Thus the Positive Grassmannian $\Grp$ is $S$-invariant, and the loss function $E(x)$ is $S$-invariant too:
\begin{equation}
    E(S x) = E(x).
\end{equation}
This observation suggests one looks for $S$-invariant subspaces that belong to $\Grp$, and checks whether they are extremal. \cite{karp} proves that, in particular, the subspace $[C]$ is the only $S$-invariant subspace on $\Grp$ and, in this paper, we prove its extremality.
\end{remark}

This operator view naturally suggests a candidate extremal point for the minimax problem \eqref{extremal_grplus} on the general Grassmannian $\GrpK{k,n}$. For any integer $k<n$, \cite{karp} defines the cyclic shift operator $\sigma$ that, for $k=2$, is the conjugate of our operator $S$.
The action of the operator $\sigma$ on the $k$-dimensional subspaces induces the (left) cyclic shift on the Plucker coordinates, and therefor both the Positive Grassmannian and the loss function $E(x)$ are $\sigma$-invariant. Karp proves (Theorem 1.1) that the operator $\sigma$ has a unique fixed point $V_0$ on the Positive Grassmannian $\GrpK{k,n}$.
\begin{conjecture}
The subspace $V_0$ minimizes $E(x)$ on $\GrpK{k,n}$.
\end{conjecture}

\section{Extremality}
The cyclic symmetry is a crucial ingredient in the study of the positive Grassmannian; see \cite{postnikov}, Sections 2, 6. This problem is no exception.

Let $\sigma$ denote the \emph{right cyclic shift} permutation of $[n]$:
\begin{equation}
\label{D:sigma_action_permutation}
    \sigma \, i =
    \begin{cases}
        i+1 & , i<n \\
        1 & , \text{otherwise},
    \end{cases}
\end{equation}
where $i \in [n]$.
Let $k$ be an integer, we extend the definition of $\sigma$ to let it to act on the set of ordered $k$-tuples $\binom{[n]}{k}$. For an ordered $k$-tuple $(i_1, \dots, i_k)$, by definition,
\begin{equation}
\label{D:sigma_action_ktuples}
    \sigma \, (i_1, \dots, i_k) =
    \begin{cases}
        (i_1+1, \dots, i_k+1) & , i_k<n \\
        (1, i_1+1, \dots, i_{k-1}+1) & , \text{otherwise}.
    \end{cases}
\end{equation}
It is straightforward to verify that the map $\sigma$ maps ordered $k$-tuples onto ordered ones.
Note also that for $1$-tuples, the above definition agrees with the $\sigma$-action on $[n]$.

\begin{proposition}
\label{P:sigma_group_action}
The map $\sigma$ induces an action of the cyclic group $\Z_n$ on the set of ordered $k$-tuples $\binom{[n]}{k}$,
i.e. $\sigma^n$ is the identity map.
\begin{proof}
Let $\sigma_{ew}$ denote the element-wise application of the \emph{permutation} $\sigma$ to
the elements of an unordered $k$-tuple:
\begin{equation*}
    \sigma_{ew} (i_1, \dots, i_k) = (\sigma i_1, \dots, \sigma i_k),
\end{equation*}
and let $sort$ denote the map that sorts an unordered $k$-tuple in the ascending order.
By the definition, the map $\sigma$ is the composition of the element-wise map $\sigma_{ew}$ and the $sort$ map:
\begin{equation*}
    \sigma = sort \cdot \sigma_{ew}.
\end{equation*}
Note that for any $k$-tuple argument $I$, the output $k$-tuple $sort(I)$ is a permutation of $I$:
\begin{equation*}
    sort \, I = p \, I,
\end{equation*}
where the permutation $p$ depends on the argument $I$.

Fix an ordered $k$-tuple $I$, and compute $\sigma^n \, I$:
\begin{align*}
    \sigma^n \, I &= (sort \cdot \sigma_{ew})^n \, I \\
    &= \prnths*{\prod_k^n p_k \cdot \sigma_{ew}} \, I && \text{, where the permutations $p_k$ ultimately depend on the argument $I$.} \\
    &= \prod_k p_k \cdot \sigma_{ew}^n \, I && \text{, since any permutation commutes with $\sigma_{ew}$.} \\
    &= sort \, I && \text{, since $\sigma_{ew}^n$ is the identity map; since the result is sorted by line 1.} \\
    &= I && \text{, since $I$ is ordered by assumption.}
\end{align*}
\end{proof}
\end{proposition}

It is known that the Plucker coordinates satisfy the set of quadratic equations,
called the Plucker relations,
\begin{equation}
\label{D:plucker_identity}
    \PluckerEq{i}{j}{k}{l},
\end{equation}
where $(i,j,k,l) \in \binom{[n]}{4}$; see \cite{miller_sturmfels} Eq. 14.3.
We refer to $\binom{[n]}{4}$ as the \emph{index set} of the Plucker relations.
\begin{remark}
A short self-contained introduction to the Grassmannian is given in \cite{smirnov}, where the equation \eqref{D:plucker_identity} appears in Proposition 2.10.
\end{remark}

For $k=2$, the map $\sigma$ acts on the index set of Plucker coordinates $\binom{[n]}{2}$, and for $k=4$, on the index set of Plucker relations. Our next proposition shows that these two actions are consistent with each other. 

\newcommand*{\PluckerEqShifted}[4]{\PluckerEqTuples{\sigma(#1,#3)}{\sigma(#2,#4)}{\sigma(#1,#2)}{\sigma(#3,#4)}{\sigma(#1,#4)}{\sigma(#2,#3)}}
\newcommand*{\PluckerEqShiftedPower}[4]{\PluckerEqTuples{\sigma^m(#1,#3)}{\sigma^m(#2,#4)}{\sigma^m(#1,#2)}{\sigma^m(#3,#4)}{\sigma^m(#1,#4)}{\sigma^m(#2,#3)}}
\begin{proposition}
\label{P:sigma_action_consistent}
For an ordered $4$-tuple $(i,j,k,l)$, the Plucker relation for the shifted $k$-tuple $\sigma (i,j,k,l)$,
\emph{up to the order} of the terms in the right-hand side, is given by
\begin{equation}
\label{E:PluckerEqShifted}
    \PluckerEqShifted{i}{j}{k}{l}.
\end{equation}
\begin{remark}
In \cite{postnikov} Section 6, Postnikov provides a cyclically invariant definition of the Positive Grassmannian.
Our proposition, and its generalization to the general Grassmannian, naturally follow from his results. For our elementary case of $\Grp$, though, a proof based on an honest computation will do.
\end{remark}
\begin{proof}
\newcommand{\p}[1]{#1^\prime}
\newcommand{\monom}[4]{\D{(#1,#2)} \D{(#3,#4)}}
Let $(\p{i},\p{j},\p{k},\p{l}) = \sigma (i,j,k,l)$. Consider the Plucker relation for $(\p{i},\p{j},\p{k},\p{l})$:
\begin{equation}
    \label{E:plucker_identity_prime}
    \PluckerEq{\p{i}}{\p{j}}{\p{k}}{\p{l}},
\end{equation}
and compute its three monomials.
\newcommand{\monoshift}[4]{\D{\p{#1},\p{#2}} \D{\p{#3},\p{#4}} = \D{#1+1,#2+1} \D{#3+1,#4+1} = \D{\sigma (#1,#2)} \D{\sigma (#3,#4)}}
Consider two cases:
\begin{itemize}
    \item $l < n$, then $(\p{i},\p{j},\p{k},\p{l}) = (i,j,k,l) + 1$.
    Since the relative order of $\p{i},\p{j},\p{k},\p{l}$ didn't change,
    for each of the monomials, from left to right, we have
    $$ \monoshift{i}{k}{j}{l},$$
    $$ \monoshift{i}{j}{k}{l},$$
    $$ \monoshift{i}{l}{j}{k}.$$
    
    \item $l = n$, then $(\p{i},\p{j},\p{k},\p{l}) = (1,i+1,j+1,k+1)$.
    The left-hand side monomial remains invariant:
    $$\D{\p{i},\p{k}} \D{\p{j},\p{l}} = \D{1,j+1} \D{i+1,k+1} = \D{\sigma(j,n)} \D{\sigma (i,k)} =  \D{\sigma (i,k)} \D{\sigma(j,l)},$$
    while the two monomials on the right-hand side swap places:
    \begin{equation*}
        \D{\p{i},\p{j}} \D{\p{k},\p{l}} = \D{1,i+1} \D{j+1,k+1}
        = \D{\sigma(i,n)} \D{\sigma (j,k)} = \D{\sigma(i,l)} \D{\sigma (j,k)},
    \end{equation*}
    \begin{equation*}
        \D{\p{i},\p{l}} \D{\p{j},\p{k}} = \D{1,k+1} \D{i+1,j+1}
        = \D{\sigma(k,n)} \D{\sigma (i,j)} = \D{\sigma(k,l)} \D{\sigma (i,j)}.
    \end{equation*}
\end{itemize}
In any case, when we substitute the above expressions for the primed monomials into the Plucker relation \eqref{E:plucker_identity_prime}, we get \eqref{E:PluckerEqShifted}.
\end{proof}
\end{proposition}

\begin{corollary}
\label{C:plucker_shifted_power}
Let $(i,j,k,l) \in \binom{[n]}{4}$. The identity
\begin{equation}
\PluckerEqShiftedPower{i}{j}{k}{l}
\end{equation}
holds for all integer $m$.
\end{corollary}

\begin{diary}
\begin{example}
For $n=4$, there is just one Plucker relation. The identities of Corollary \ref{C:plucker_shifted_power}
present it in four different ways:
\begin{gather*}
    \D{1,3} \D{2,4} = \D{1,2} \D{3,4} + \D{1,4} \D{2,3}, \\
    \D{2,4} \D{1,3} = \D{2,3} \D{1,4} + \D{1,2} \D{3,4}, \\
    \D{1,3} \D{2,4} = \D{3,4} \D{1,2} + \D{2,3} \D{1,4}, \\
    \D{2,4} \D{1,3} = \D{2,3} \D{1,4} + \D{1,2} \D{3,4}.
\end{gather*}
Note how the order of the terms $\D{1,2} \D{3,4}$ and $\D{1,4} \D{2,3}$ alternates from line to line.
This alternation plays a crucial role in the proof of Lemma \ref{L:geomeanJKL}.
\end{example}
\end{diary}

The $\sigma$-action \eqref{D:sigma_action_ktuples} on $\binom{[n]}{2}$ stratifies the set of Plucker coordinates into orbits. By Proposition \ref{P:sigma_group_action}, the multiset
\begin{equation}
\label{D:O_k}
    O_k = \set{\sigma^m (1, k+1), m \in [0,n-1]},
\end{equation}
is $\sigma$-invariant, where $k \in [n-1]$, and where $[a,b]$ denotes the range of integers from $a$ to $b$. We shall refer to the multiset $O_k$ as the \emph{$k$-th orbit} of the Plucker coordinates.

Let
\begin{equation}
\label{D:d}
    d = \floor{n/2}.
\end{equation}
We prove there are exactly $d$ distinct orbits. The \emph{outer} orbits $O_1$ and $O_d$ play a special role: for the extremal points of the loss function $E(x)$, they house the smallest and the largest minors respectively; see Theorem \ref{T:main}.

\newcommand{\geomean}[1]{\left(\prod_{m=0}^{n-1} #1 \right)^{1/n}}

For a point $x \in \Grp$, let $D_k$ denote the \emph{geometric mean} of its coordinates $\D{i,j}$ over the $k$-th orbit:
\begin{equation}
\label{D:D_k}
    D_k = \geomean{\D{\sigma^m(1,k+1)}}.
\end{equation}

Let
\begin{equation*}
    k \rightarrow \bar{k} = n - k
\end{equation*}
be an involution of $[n-1]$.

\begin{lemma} 
\label{L:orbits_structure}
\hfill
\begin{enumerate}
    \item\label{I:orbits_involution}
    The orbits $O_k$, and the geometric means $D_k$ are invariant under the involution:
    $$O_{\bar{k}} = O_k \text{, and } D_{\bar{k}} = D_k.$$
    \item\label{I:orbits_are_distinct} The orbits $O_k$ are distinct for $k \in [d]$.
    \item\label{I:orbits_involution_union} If $(i,j) \in O_k$, then either $j-i=k$, or $j-i=\bar{k}$.
\end{enumerate}
\end{lemma}
\begin{proof}
For item \eqref{I:orbits_involution}, fix $k$ and consider the index $I = (n-k,n)$.
By the definition of $O_k$, the index $I$ belongs to $O_k$.
Now apply the map $\sigma$ to the index $I$:
$$\sigma I = (1, (n-k) + 1) = (1, \bar{k} + 1).$$
By the same definition, the index $\sigma I \in O_{\bar{k}}$.
Since, by the definition of the orbit, the indexes $I$ and its image $\sigma I$ belong to the same orbit, the orbits $O_k$ and $O_{\bar{k}}$ must coincide.
The second equality $D_{\bar{k}} = D_k$ follows immediately from $O_{\bar{k}} = O_k$.

For item \eqref{I:orbits_are_distinct},
consider the function
\newcommand{\s}[1]{\sin #1 \frac{\pi}{n}}
\begin{equation}
    s(i,j) = \s{(j-i)}.
\end{equation}
Since $\s{k} = \s{\bar{k}}$, the function $s(i,j)$ is invariant over each orbit $O_k$.
Since $\sin(x)$ is monotonic on the interval $[0,\pi/2]$, the numbers $\s{1}, \dots, \s{d}$ are distinct, and therefor their respective orbits are distinct too.

For item \eqref{I:orbits_involution_union}, fix $k \in [d]$.
Since the function $s(i,j)$ is $\sigma$-invariant,
\begin{equation*}
    \s{(j-i)} = s(i,j) = s(1,k+1) = \s{k},
\end{equation*}
for all $(i,j) \in O_k$.
Since the equation $\sin x = \sin \phi$ has exactly two solutions $x=\phi$ and $x=\pi - \phi$ on the interval $[0,\pi]$,
\begin{equation*}
    (j-i)\frac{\pi}{n} = \text{ either } k \frac{\pi}{n} \text{,  or  } \pi - k \frac{\pi}{n}.
\end{equation*}
The first case of the identity implies $j-i=k$, while the second $j-i=\bar{k}$.

\end{proof}

Let
\begin{equation}
\label{D:s_k}
    s_k = \sin k \frac{\pi}{n}.
\end{equation}
For a point $x \in \Grp$, the coordinates $\D{i,j}$ are defined up to a common scaling factor. We now wish to make use of this degree of freedom and, \emph{from this point on}, we assume that $\D{i,j}$ are scaled so that
\begin{equation}
\label{E:normalization_condition_D_k}
    D_1 = D_{n-1} = s_1,
\end{equation}
which is possible since $D_1 = D_{n-1}$ by Proposition \ref{L:orbits_structure}, item \eqref{I:orbits_involution}.
We refer to \eqref{E:normalization_condition_D_k} as the \emph{normalization condition}. Note the cyclic matrix $C$ satisfies the normalization condition.

The following proposition lists a number of special properties of the cyclic matrix $C$. These properties ensure that $C$ satisfies \emph{exactly} certain inequalities we develop. And this exactness, in turn, will lead to optimality.

\begin{proposition}
\label{P:properties_of_C}
The Plucker coordinates $\D{i,j}$ of $C$, and their geometric means $D_k$ satisfy the following properties.
\begin{enumerate}
    \item 
Coordinates $\D{i,j}$ are constant over the $\sigma$-orbits and, therefor, are equal to their geometric means:
\begin{equation*}
\label{I:dk_C}
        \D{\sigma^m(1,k+1)} = s_k = D_k,
\end{equation*}
for all $k \in [n-1]$, and $m \in [0,n-1]$.
\newcommand{\p}{\frac{s_j s_{l-k}}{s_k s_{l-j}}}
\newcommand{\q}{\frac{s_l s_{k-j}}{s_k s_{l-j}}}
\item The weights property:
\begin{equation*}
\label{p_and_q}
    \p + \q = 1,
\end{equation*}
for all $(j,k,l) \in \binom{[n-1]}{3}$.
\item
\label{I:EC_sd_over_s1}
The sequence $s_k$ is strictly monotonically increasing for $k \in [d]$, and thus
\begin{gather*}
    \max \D{i,j} = s_d, \\
    \min \D{i,j} = s_1, \\
    E(C) = \frac{s_d}{s_1}.
\end{gather*}
\end{enumerate}

\begin{proof} Item by item:
\begin{enumerate}
    \item Fix $(i,j) \in O_k$, and compute $\D{i,j}$. By basic trigonometry,
    \begin{align*}
        \D{i,j} &= \begin{vmatrix}
        \cosk i & \cosk j \\
        \sink i & \sink j
        \end{vmatrix}
        = \sink{(j-i)} \\
        &= \sink{k} && \text{, by Lemma \ref{L:orbits_structure} item \eqref{I:orbits_involution_union}}.
    \end{align*}
    Since $\D{i,j}$ are constant over $O_k$, their geometric mean $D_k$ equals to the constant.
    
    \item Fix $(j,k,l) \in \binom{[n-1]}{3}$, and consider the $4$-tuple $I = (1,j+1,k+1,l+1)$.
    Substitute $\D{a,b} = s_{b-a}$ for each Plucker coordinate in the Plucker identity
    \eqref{D:plucker_identity} for the $4$-tuple $I$:
    \begin{equation*}
        s_{(k+1) - 1} s_{(l+1) - (j+1)} = s_{(j+1) - 1} s_{(l+1) - (k+1)} + s_{(l+1) - 1} s_{(k+1) - (j+1)},
    \end{equation*}
    which simplifies to the identity 
    \begin{equation*}
        s_k s_{l-j} = s_j s_{l-k} + s_l s_{k-j}.
    \end{equation*}
    Divide each side by $s_k s_{l-j}$ to get the weights property.
    
    \item This item is self-evident.
\end{enumerate}
\end{proof}
\end{proposition}

Our next lemma is the main reductive step in our proof; it says that the geometric averaging (\ref{D:D_k}) morphs the system (\ref{D:plucker_identity}) of quadratic identities for $\D{i,j}$ into a similar system of quadratic \emph{inequalities} for $D_k$.
\begin{lemma}
For $x \in \Grp$, the geometric means $D_k$ satisfy
\newcommand{\PluckerInq}[3]{D_{#2} D_{#3-#1} \geq D_{#1} D_{#3-#2} + D_#3 D_{#2-#1}}
\label{L:geomeanJKL}
\begin{equation}
\label{E:geomeanJKL}
    \PluckerInq{j}{k}{l},
\end{equation}
for all $(j, k, l) \in \binom{[n-1]}{3}$.

\begin{proof}
Fix $(j, k, l) \in \binom{[n-1]}{3}$ and consider the 4-tuple $(1, j+1, k+1, l+1)$.
By Corollary \ref{C:plucker_shifted_power}, the identities
\begin{equation}
    \PluckerEqShiftedPower{1}{j+1}{k+1}{l+1}
\end{equation}
hold for all $m \in [0,n-1]$.
Take the geometric mean of each side of the above $n$ identities:
\begin{equation*}
    \geomean{\D{\sigma^m(1,k+1)} \D{\sigma^m(j+1,l+1)}}
    = \geomean{\left(\D{\sigma^m(1,j+1)} \D{\sigma^m(k+1,l+1)} + \D{\sigma^m(1,l+1)} \D{\sigma^m(j+1,k+1)}\right)}.
\end{equation*}
For the left-hand side,
\begin{equation*}
    \geomean{\D{\sigma^m(1,k+1)} \D{\sigma^m(j+1,l+1)}}
    = \geomean{\D{\sigma^m(1,k+1)}} \geomean{\D{\sigma^m(j+1,l+1)}}
    = D_k D_{l-j},
\end{equation*}
since the index set in the second product,
\begin{equation*}
    \set{\sigma^m (j+1, l+1) \text{, for } m \in [0,n-1]},
\end{equation*}
is $\sigma$-invariant by Proposition \ref{P:sigma_group_action}, and therefor equals to $O_{l-j}$.
For the right-hand side,
apply the superadditivity inequality
\begin{equation*}
    \geomean{(a_m + b_m)} \geq \geomean{a_m} + \geomean{b_m}
\end{equation*}
(\cite{steele}, Exercise 2.1) to the right-hand side:
\begin{multline*}
    \geomean{\left(\D{\sigma^m(1,j+1)} \D{\sigma^m(k+1,l+1)} + \D{\sigma^m(1,l+1)} \D{\sigma^m(j+1,k+1)}\right)} \\
    \geq \geomean{\D{\sigma^m(1,j+1)} \D{\sigma^m(k+1,l+1)}}
    + \geomean{\D{\sigma^m(1,l+1)} \D{\sigma^m(j+1,k+1)}} \\
    = \geomean{\D{\sigma^m(1,j+1)}} \geomean{\D{\sigma^m(k+1,l+1)}} \\
    + \geomean{\D{\sigma^m(1,l+1)}} \geomean{\D{\sigma^m(j+1,k+1)}} \\
    = D_j D_{l-k} + D_l D_{k-j}.
\end{multline*}
\end{proof}
\end{lemma}

\begin{diary}
\begin{corollary}
For $(j,k) \in \binom{[n-1]}{2}$, $j+k \leq n$
\begin{equation}
\label{geomeanJK}
    D_k^2 \geq D_j^2 + D_{k+j} D_{k-j}
\end{equation}
\begin{proof}
Set $l=j+k$ in (\ref{E:geomeanJKL}).
\end{proof}
\end{corollary}

\begin{corollary}
Let $d=\floor{n/2}$. The sequence $D_j$ is strictly monotonically increasing for $j = 1, \dots, d$.
\begin{proof}
Fix $j < d$, set $k=j+1$. Observe $k \leq d$ and therefor $j+k < 2d \leq n$. Apply (\ref{geomeanJK}).
\end{proof}
\end{corollary}

\end{diary}

\begin{diary}
\begin{example}
For $n=5$ there are 4 orbits that come in pairs $O_1=O_4$ and $O_2=O_3$. Respectively their geometric means $D_k$ come in pairs $D_1=D_4$ and $D_2=D_3$. The number of distinct orbits is $d=\floor{5/2}=2$.
The sequence $D_1,D_2,D_3,D_4$ satisfy four inequalities, one for each $(j,k,l) \in \binom{[4]}{3}$, see \eqref{E:geomeanJKL},
\begin{align*}
    D_2 D_2 \geq D_1 D_1 + D_3 D_1 && (1,2,3) \\
    D_2 D_3 \geq D_1 D_2 + D_4 D_1 && (1,2,4) \\
    D_3 D_3 \geq D_1 D_1 + D_4 D_2 && (1,3,4) \\
    D_3 D_2 \geq D_2 D_1 + D_4 D_1 && (2,3,4)
\end{align*}
When we substitute $D_3=D_2$ and $D_4=D_1$, all the four inequalities collapse into a single one.
This collapse into a single equation is a little bit misleading, it simplifies matters for $n=5$, and in fact leads to a quick proof, but makes things intractable for larger $n$.
$$ D_2^2 \geq D_1^2 + D_1 D_2$$
\end{example}

\end{diary}

Let $D_k$, where $k \in [n-1]$, be a sequence of positive numbers. Define their \emph{normalized logs}
\begin{equation}
\label{D:a_k}
    a_k = \log \frac{D_k}{s_k},
\end{equation}
where $s_k$ are given by \eqref{D:s_k}.

\begin{lemma}
\label{L:linear_ineqs_a_k}
\newcommand{\p}{\frac{s_j s_{l-k}}{s_k s_{l-j}}}
\newcommand{\q}{\frac{s_l s_{k-j}}{s_k s_{l-j}}}

If $D_k$ satisfy the quadratic inequalities \eqref{E:geomeanJKL}, then $a_k$ satisfy the system of linear inequalities
\begin{equation}
\label{lin_general}
    a_k + a_{l-j} \geq \p (a_j + a_{l-k})
    + \q (a_l + a_{k-j}),
\end{equation}
for all $(j,k,l) \in \binom{[n-1]}{3}$.
\begin{proof}
Let $d_k = D_k/s_k$. Substitute $D_k=s_k d_k$ into the quadratic inequalities (\ref{E:geomeanJKL}), and divide both sides by $s_k s_{l-j}$:
\begin{equation*}
    d_k d_{l-j} \geq \p d_j d_{l-k} + \q d_l d_{k-j}.
\end{equation*}
Take the $\log$, recall the weights property (Proposition \ref{P:properties_of_C}, item \eqref{p_and_q}),
and use the concavity to get the lower bound:
\begin{multline*}
    a_k + a_{l-j} = \log d_k d_{l-j} \geq \log \prnths*{\p d_j d_{l-k} + \q d_l d_{k-j}}{} \\
    \geq \p \log (d_j d_{l-k}) + \q \log (d_l d_{k-j}) \\
    = \p (a_j + a_{l-k}) + \q (a_l + a_{k-j})
\end{multline*}
\end{proof}
\end{lemma}

\begin{diary}
\begin{corollary}
\newcommand{\p}{\frac{s_j^2}{s_k^2}}
\newcommand{\q}{\frac{s_{k+j} s_{k-j}}{s_k^2}}
Let
\begin{gather}
p_{j,k} = \p \\
q_{j,k} = \q
\end{gather}
For $(j,k) \in \binom{[n-2]}{2}$, $j+k \leq n-1$
\begin{equation}
\label{lin_jk}
    a_k \geq p_{j,k} a_j + q_{j,k} \frac{a_{k+j} + a_{k-j}}{2}
\end{equation}
\begin{proof}
Set $l = j + k$ in (\ref{p_and_q}) and in (\ref{lin_general}).
\end{proof}
\end{corollary}

\end{diary}

\newcommand{\p}{\frac{s_{1}^2}{s_k^2}}
\newcommand{\q}{\frac{s_{k+1} s_{k-1}}{s_k^2}}

\begin{corollary}
\label{C:a_k_are_sub_concave}
If additionally $D_k$ satisfy the boundary condition  $D_1=D_{n-1}=s_1$,
then $a_k$ satisfy the boundary condition
\begin{equation}
\label{E:boundary_condition_a_k}
    a_1 = a_{n-1} = 0,
\end{equation}
and for every $k \in [2,n-2]$, there exists $0 < q_k < 1$, such that 
\begin{equation}
\label{lin_k}
    a_k \geq q_{k} \frac{a_{k+1} + a_{k-1}}{2}.
\end{equation}

\begin{proof}
Fix $k \in [2,n-2]$, and substitute $j=1$ and $l = k + 1$ in the weights property (\ref{p_and_q}) of Proposition \ref{P:properties_of_C}. The weights property simplifies to
\begin{equation}
    p_k + s_k = 1,
\end{equation}
where
\begin{align}
    p_k = \p, & \text{\phantom{and}} q_k = \q.
\end{align}
Since $p_k$ and $q_k$ are strictly positive and sum up to one, $q_k < 1$.

Substitute $j=1$ and $l = k+1$ in \eqref{lin_general}:
\begin{equation*}
    a_k + a_k \geq p_k \prnths{a_1 + a_1} + q_k \prnths{a_{k+1} + a_{k-1}}.
\end{equation*}
Divide both sides by $2$, and recall that $a_1=0$:
\begin{equation*}
    a_k \geq q_k \frac{a_{k+1} + a_{k-1}}{2}.
\end{equation*}
\end{proof}
\end{corollary}

\begin{corollary}
\label{C:a_k_are_positive}
Under the assumptions of Corollary \ref{C:a_k_are_sub_concave}
\begin{enumerate}
    \item  $a_k \geq 0$,
    \item if \emph{any} $a_k > 0$, then \emph{all} $a_k > 0$,
\end{enumerate}
for all $k \in [2,n-2]$.
\begin{proof}
\newcommand{\Czero}[1]{\R_0^{#1}}
We continue to use the notation from Corollary \ref{C:a_k_are_sub_concave}.
Let $\Czero{n-1}$ be the subspace of $\R^{n-1}$ that satisfy the boundary condition $x_1 = x_{n-1} =0$.
On $\Czero{n-1}$ define the linear operator $S: (x_k) \rightarrow (y_k)$:
\begin{align*}
    y_k &= q_k \frac{x_{k+1} + x_{k-1}}{2} \text{ , } 2 \leq k \leq n-2, \\
    y_1 &= y_{n-1} = 0.
\end{align*}
Note that if $x \geq 0$ (coordinate-wise), then $S x \geq 0$, where $x \in \Czero{n-1}$.
Therefor if $u \geq v$, for $u,v \in \Czero{n-1}$, then $S u \geq S v$.

Denote $a = (a_j)$. Since $a_1 = a_{n-1} = 0$ by Corollary \ref{C:a_k_are_sub_concave}, the vector $a$ is in $\Czero{n-1}$. We reinterpret the set of inequalities (\ref{lin_k}) as a vector inequality (coordinate-wise):
\begin{equation}
\label{s_pointwise}
    a \geq S a,
\end{equation}
and iterate it $m$ times:
\begin{equation}
\label{s_iterated}
    a \geq S^m a.
\end{equation}
By the definition of $S$, its norm $\norm{S}_{\infty}$ is strictly less than one:
\begin{equation*}
    \norm{S}_{\infty} \leq \max q_k < 1,
\end{equation*}
and therefor the right-hand side in (\ref{s_iterated}) converges to zero as $m \rightarrow \infty$. Passing to the limit, we get
\begin{equation*}
    a \geq 0.
\end{equation*}

We now prove the second part of the claim by bootstrapping the fact $a_k \geq 0$, for all $k$.
Let $a_k > 0$ for some inner index $1 < k < n-1$. By the definition of $S$, the $k \pm 1$ coordinates of $S a$ are positive  whenever $k \pm 1$ is itself an inner index; then (\ref{s_pointwise}) implies that each $a_{k \pm 1} > 0$. We can say strict positivity \emph{diffuses} to neighbouring indexes. Iterating this argument we obtain $a_k > 0$, for all inner index $k$.
\end{proof}
\end{corollary}

Let $x \in \Grp$, and let again $D_k$ be the geometric mean of $\D{i,j}$ over the $k$-th orbit; see  \eqref{D:D_k}. Define the auxiliary loss function
\begin{equation}
    L(x) = \frac{D_d}{D_1}.
\end{equation}
Our interest in $L(x)$ is due to its relation to $E(x)$.
\begin{proposition}\label{P:L_leq_E}
\hfill
\begin{enumerate}
    \item\label{I:L_weakly_smaller_than_E} The loss function $L(x)$ is weakly smaller than $E(x)$:
    \begin{equation*}
        L(x) \leq E(x),
    \end{equation*}
    \item\label{I:L_and_E_coincide_on_C} $L(x)$ and $E(x)$ coincide on $C$:
    $$E(C)=L(C) = \frac{s_d}{s _1}.$$
\end{enumerate}
\begin{proof}
For item \eqref{I:L_weakly_smaller_than_E},
\begin{equation*}
    E(x) = \frac{\max\limits_{(i,j) \in \binom{[n]}{2}} \D{i,j}}{\min\limits_{(i,j) \in \binom{[n]}{2}} \D{i,j}}
    \geq \frac{\max\limits_{(i,j) \in O_d} \D{i,j}}{\min\limits_{(i,j) \in O_1} \D{i,j}}
    \geq \frac{D_d}{D_1}
    = L(x)
\end{equation*}

Item \eqref{I:L_and_E_coincide_on_C} follows from Proposition \ref{P:properties_of_C}.

\end{proof}
\end{proposition}

\begin{theorem}
\label{T:C_minimizes_E}
The circular matrix $C$ minimizes $E(x)$ on $\Grp$.
\begin{proof}
Let $x \in \Grp$, and let $a_k = \log \frac{D_k}{s_k}$ denote the normalized logs of its geometric means $D_k$; see \eqref{D:a_k}.
By Lemma \ref{L:geomeanJKL}, the geometric means $D_k$
satisfy the inequalities \eqref{E:geomeanJKL}. Additionally $D_1 = D_{n-1} = s_1$ by the normalization condition \eqref{E:normalization_condition_D_k}. Therefor Corollary \ref{C:a_k_are_positive} applies, and $a_d \geq 0$. We now show that $a_d \geq 0$ implies $C$ minimizes $L(X)$:
\begin{equation*}
\begin{aligned}
    \log L(x) &= \log D_d - \log D_1 \\
    &= \log s_d + a_d - \log s_1 &&  \text{, since $D_1 = s_1$, see  \eqref{E:normalization_condition_D_k}.}\\
    &= \log \frac{s_d}{s_1} + a_d \\
    &\geq \log L(C) && \text{, since $a_d \geq 0$.}
\end{aligned}
\end{equation*}
By Proposition \ref{P:L_leq_E},
\begin{equation*}
    E(C) = L(C) \leq L(x) \leq E(x).
\end{equation*}

\end{proof}
\end{theorem}

\subsection{Absolute Values and The Polygon}

In their original version of the loss function \eqref{D:Ex} \cite{berman} used the absolute values of $\D{i,j}$:
\begin{equation}
    B(X) = \frac{\max \abs{\D{i,j}}}{\min \abs{\D{i,j}}},    
\end{equation}
and wanted to minimize $B(X)$ over the entire $\Gr$. It turns out the addition of the absolute values doesn't change the problem in any significant way.

\begin{proposition}
\label{P:B_E_loss_equivalence}
For every $2 \times n$ matrix $X$, there exists a matrix $Y$, such that the set of its Plucker coordinates $\set{\D{i,j}(Y)}$ coincides with the set of absolute values $\set{\abs{\D{i,j}}}$ of matrix $X$.

\begin{proof}
Consider the columns $x_i =(u_i, v_i)$ of the matrix $X$ as planar vectors. For those vectors that are not already in the upper half-plane, flip their signs to put them there. Algebraically:
\begin{equation}
    x_i = (\sign v_i) \cdot x_i.
\end{equation}
This transformation possibly flips signs of $\D{i,j}$, but doesn't change their absolute values.

Now sort the columns of $X$ counterclockwise, and denote the sorted matrix $Y$. Since the columns of $Y$ are in the same half-plane and sorted counterclockwise, all its minors $\D{i,j}(Y)$ are non-negative and, by construction, equal to the absolute values of their counterparts in $X$.
\end{proof}
\end{proposition}

If we recall that the absolute value of the determinant equals to the area of the respective parallelogram, then the loss function $B(X)$ and its extremal value $C$ have a pleasant planar geometry interpretation.
For a polygon $X$ with vertices $x_i$, the value $B(X)$ is the ratio of the maximum-area triangle $(x_i,O,x_j)$ and the minimum one, where $O$ is the origin of the plane.
The polygon defined by the cyclic matrix $C$ is the upper half of the regular $2n$ polygon,
and by Theorem \ref{T:C_minimizes_E} and Proposition \ref{P:B_E_loss_equivalence}, it minimizes $B(X)$.

\section{Uniqueness}

In Theorem \ref{T:main} we proved that the span $[C]$ of the cyclic matrix $C$, defined by the equation \eqref{D:cyclic_matrix}, is an extremal point. In this section we discuss when it is the only one.

Our next theorem shows that, for any extremal point, the structure of its Plucker coordinates $\D{i,j}$ is similar to that of $C$, but stops short of implying uniqueness for all $n$. And, in fact, we show the uniqueness fails for any $n \mod 4 = 2$.
\begin{theorem}
\label{T:main}
If $x \in \Grp$ is extremal, then
\begin{enumerate}
\item\label{I:main:D_k_equal_s_k}
The geometric means $D_k$ satisfy
\begin{equation*}
    D_k = s_k,
\end{equation*}
for all $k \in [n-1]$.

\item\label{I:main:s_ij_bound}
All the coordinates $\D{i,j}$ are bounded by $s_1$ and $s_d$:
\begin{equation*}
    s_1 \leq \D{i,j} \leq s_d.
\end{equation*}{}

\item \label{I:main:D_ij_constant_over_outer_orbits}
The coordinates $\D{i,j}$ are constant over each of the two outer orbits $O_1$ and $O_d$:
\begin{equation*}
    \D{i,j} = s_k,
\end{equation*}
for all $(i,j) \in O_k$, where $k=1,d$.
\end{enumerate}
\end{theorem}
\begin{proof}\hfill
\begin{enumerate}
    \item Let again $a_k$ be the normalized log of the geometric mean $D_k$; see \eqref{D:a_k}.
    By the definition of $a_k$, the identity $D_k=s_k$ is equivalent to $a_k = 0$, for all $k \in [n-1]$.
    Since, by the assumption, $x$ is extremal, $E(x) = E(C)$, and therefor, by Proposition \ref{P:L_leq_E}, $D_d = s_d$, which is equivalent to $a_d=0$. By Corollary \ref{C:a_k_are_positive}, it implies all $a_k=0$.
    
    \item Since, by item \eqref{I:main:D_k_equal_s_k}, the geometric mean $D_d$ is equal to $s_d$, the maximum $M$ of the coordinates $\D{i,j}$ is weakly bigger than $s_d$:
    $$ M = \max\limits_{(i,j) \in \binom{[n]}{2}} \D{i,j} \geq \max\limits_{(i,j) \in O_d} \D{i,j} \geq D_d = s_d. $$
    Similarly, $D_1 = s_1$ implies
    $$ m = \min \D{i,j} \leq D_1 = s_1. $$
    On the other hand, since $x$ is extremal by the assumption, and $C$ is extremal by Theorem \eqref{T:C_minimizes_E},
    $$ \frac{M}{m} = E(x) = E(C) = \frac{s_d}{s_1}, $$
    which implies the weak inequalities for $m$ and $M$ must in fact be equalities.
    
    \item Consider, for example, the case $k=1$. By item \eqref{I:main:s_ij_bound},
    all $\D{i,j} \geq s_1$, where $(i,j) \in O_1$. On the other hand, their geometric mean $D_1$ equals $s_1$ by item \eqref{I:main:D_k_equal_s_k}, which implies $\D{i,j} = s_1$, for all $(i,j) \in O_1$.
\end{enumerate}
\end{proof}

\begin{proposition}
\label{P:uniqueness_n_eq_4}
For $n=4$, the span $[C]$ of the cyclic matrix $C$ is the unique extremal point.
\begin{proof}
For $n=4$, the number of unique orbits $d = \floor{n/2} = 2$; see Lemma \ref{L:orbits_structure}.
In other words, $O_1$ and $O_2$ are the only $\sigma$-orbits. Let $x \in \Grp$ be an extremal point. By Theorem \ref{T:main}, item \eqref{I:main:D_ij_constant_over_outer_orbits}, its Plucker coordinates $\D{i,j}$ are equal to those of $C$, for all $(i,j) \in \binom{[n]}{2}$. Since the Plucker coordinates uniquely identify the point (\cite{miller_sturmfels} Proposition 14.2), $x=[C]$.
\end{proof}
\end{proposition}

For any $n$, each of the orbits $O_1, \dots, O_{d-1}$ has $n$ distinct elements. The orbit $O_d$ is special: for an odd $n$ it has $n$ distinct elements, while for an even $n$ it has only $n/2$ distinct elements. It turns out that, for the odd $n$, the $2 n$ values $\D{i,j}$, where $(i,j) \in O_1 \cup O_d$, uniquely identify a point $x \in \Grp$. To prove this, we need to establish an identity that could be regarded as a vector form of the Plucker relation \eqref{D:plucker_identity}.

For two planar vectors $u,v \in \R^2$, let the \emph{wedge product} $u \wedge v$ denote the determinant of the $2\times 2$ column matrix $[u, v]$.

\begin{lemma}\label{L:uvw_identity}
Let $u,v,w \in \R^2$ be three planar vectors. The vectors $u,v,w$ satisfy the vector identity
\begin{equation}
\label{E:uvw_identity}
    (u \wedge v) w = (u \wedge w) v - (v \wedge w) u.
\end{equation}
\begin{proof}
Let $z \in \R^2$ be an arbitrary planar vector. Consider the column matrix $[u,v,w,z]$ and apply the Plucker relation to it:
$$ (u \wedge w) (v \wedge z) = (u \wedge v) (w \wedge z) + (u \wedge z) (v \wedge w), $$
where the determinants are denoted by the wedge products.
Factor out $z$ from both sides of the equation:
$$ ((u \wedge w) v) \wedge z = ((u \wedge v) w  + (v \wedge w) u) \wedge z. $$
Since $z$ is arbitrary,
$$ (u \wedge w) v = (u \wedge v) w  + (v \wedge w) u, $$
from which \eqref{E:uvw_identity} immediately follows.

\end{proof}
\end{lemma}

\begin{theorem}
\label{T:outer_orbits_identify_point}
For an odd $n$, the outer orbits coordinates $\D{i,j}$, where $(i,j) \in O_1 \cup O_d$, uniquely identify a point $x \in \Grp$.
\begin{proof}
In this proof we allow out-or-order subscripts for $\D{i,j}$; as usual $\D{i,j} = -\D{j,i}$.
Define the integer sequence $c(k)$ as the orbit of $1$ under the iterations of the permutation $\sigma^d$:
$$ c(k) = \sigma^{k d} \, 1,$$
where $k \in [0,n-1]$.
\newcommand{\xc}[1]{x_{c(#1)}}
Let $X$ be a spanning matrix for $x$, i.e. $x=[X]$, and let $x_m$ denote the $m$-th column of the matrix $X$.
Fix $k \geq 2$, and apply the identity \eqref{E:uvw_identity} to the triplet  $\xc{k-2}, \xc{k-1}, \xc{k}$:
$$ (\xc{k-2} \wedge \xc{k-1}) \xc{k} = (\xc{k-2} \wedge \xc{k}) \xc{k-1} - (\xc{k-1} \wedge \xc{k}) \xc{k-2}. $$
Since, by the definition, $x_i \wedge x_j = \D{i,j}$, the previous identity can be written as
\begin{equation}
\label{E:ck_identity}
    \D{c(k-2),c(k-1)} \xc{k} = \D{c(k-2),c(k)} \xc{k-1} - \D{c(k-1),c(k)} \xc{k-2}.
\end{equation}
Since, by the assumption, $x \in \Grp$, all $\D{i,j} \neq 0$ and we can divide \eqref{E:ck_identity} by $\D{c(k-2),c(k-1)}$:
\begin{equation}
\label{E:ck_iterative}
   \xc{k} =  \D{c(k-2),c(k-1)}^{-1} \prnths{\D{c(k-2),c(k)} \xc{k-1} - \D{c(k-1),c(k)} \xc{k-2}}.
\end{equation}

We now show that all the $\D{i,j}$ in \eqref{E:ck_iterative}, up to a sign, come from the outer orbits $O_1$ and $O_d$.
Indeed, by the definition,
$$c(k-1) = \sigma^d c(k-2).$$
If $c(k-2) < c(k-1)$, then the index $(c(k-2), c(k-1)) \in O_d$, by the definition of $O_d$, otherwise its transposition
\footnote{recall that, by the definition of the orbits $O_k$, indexes $(i,j)$ are ordered pairs, while the pair $(c(k-2), c(k-1))$ might be out of order.}
$(c(k-1), c(k-2)) \in O_d$. In either case, the value $\D{c(k-2),c(k-1)}$, up to a sign, comes from $O_d$.
Similarly, inside the parenthesis, the index $(c(k-1), c(k))$ in the second term, up to a transposition, belongs to $O_d$.
And the index $(c(k-2), c(k))$ in the first term, up to a transposition, belongs to $O_{2d} = O_1$; see Lemma \ref{L:orbits_structure}, item \eqref{I:orbits_involution}.

By starting the recurrence \eqref{E:ck_iterative} from $k=2$, and running it through $k=n$, we prove that each $x_{c(k)}$, ultimately, is a linear combination of the initial two columns $x_{c(0)}$ and $x_{c(1)}$, where the coefficients of the linear combinations are determined by the outer orbits coordinates $\D{i,j}$.

Since $\gcd(d, n) = \gcd(d, 2d+1) = 1$, the permutation $\sigma^d$ is a generator of the cyclic group $\cyclic{\sigma}$, and therefor the sequence $c(k)$ visits each $m \in [n]$.
This implies that each column $x_m$ is a linear combination of $x_{c(0)}$ and $x_{c(1)}$  with the coefficients determined by the outer orbits coordinates, and thus the span $x = [X]$ is uniquely determined by the coordinates.

\end{proof}
\end{theorem}

\begin{proposition}
\label{P:uniqueness_n_odd}
For an odd $n$, the span $[C]$ of the cyclic matrix $C$ is the unique extremal point.
\begin{proof}
If $x \in \Grp$ is extremal, its outer orbits $\D{i,j}$ coincide with those of $C$; see Theorem \ref{T:main}, item \eqref{I:main:D_ij_constant_over_outer_orbits}. By Theorem \ref{T:outer_orbits_identify_point}, the point $x$ equals $[C]$.
\end{proof}
\end{proposition}

\begin{proposition}
\label{P:non_uniqueness_n_mod_4_eq_2}
If $n \mod 4 = 2$, then there exists a continuous family of matrices $C^q$, where $q > 0$, such that
\begin{enumerate}
    \item $C^1 = C$.
    \item\label{I:Cq_row_spans_distinct} the row spans $[C^q]$ are distinct for all $q > 0$.
    \item\label{I:ECq_equals_EC} $E(C^q) = E(C)$, for all $q$ in a sufficiently small neighborhood of $1$.
\end{enumerate}

\begin{proof}
Let $c_i$ denote the $i$-th column of the cyclic matrix $C$.
Fix $q>0$, and define the \emph{$q$-transform} that multiplies odd-numbered columns by $1/q$ and the even-numbered ones by $q$:
$$ c_i^q = q^{(-1)^i} c_i.$$
Define the matrix $C^q$:
$$ C^q = [c_1^q, \dots, c_n^q]. $$
Note that $C^1 = C$.
Let $\D{i,j}$ and  $\D{i,j}^q$ denote the Plucker coordinates of the cyclic matrix $C$ and $C^q$ respectively.
Compute $\D{i,j}^q$:
$$ \D{i,j}^q = c_i^q \wedge c_{j}^q = q^{(-1)^i + (-1)^j} \D{i,j} = q^{(-1)^i(1 + (-1)^{j-i})} \D{i,j}.$$
Fix $k \in [n-1]$, and examine the above identity for $(i,j) \in O_k$.
If $(i,j) \in O_k$, then by Lemma \ref{L:orbits_structure}, item \eqref{I:orbits_involution_union},
either $j-i=k$, or $j-i=\bar{k}$. If $j-i=k$, then
\begin{equation}
\label{E:non_uniqueness_n_mod_4_eq_2:k}
     \D{i,j}^q = q^{(-1)^i (1 + (-1)^k)} \D{i,j}.
\end{equation}
If $j-i=\bar{k}$, then the previous identity still holds:
\begin{equation}
\label{E:non_uniqueness_n_mod_4_eq_2:k_bar}
\begin{aligned}
     \D{i,j}^q &= q^{(-1)^i (1 + (-1)^{\bar{k}})} \D{i,j} \\
     &= q^{(-1)^i (1 + (-1)^k)} \D{i,j} && \text{, since } (-1)^{\bar{k}} = (-1)^k \text{ for even } n.
\end{aligned}
\end{equation}
Combine the two identities \eqref{E:non_uniqueness_n_mod_4_eq_2:k} and \eqref{E:non_uniqueness_n_mod_4_eq_2:k_bar} into a single identity for $O_k$:
\begin{equation}
\label{E:D_ij_on_k_orbit}
     \D{i,j}^q = q^{(-1)^i (1 + (-1)^k)} \D{i,j},
\end{equation}
where $(i,j) \in O_k$.

For the odd $k$, equation \eqref{E:D_ij_on_k_orbit} simplifies to
\begin{equation}
\label{E:D_ij_on_odd_orbit}
     \D{i,j}^q = \D{i,j},
\end{equation}
and for the even $k$, to
\begin{equation}
\label{E:D_ij_on_even_orbit}
     \D{i,j}^q = q^{2 (-1)^i} \D{i,j}.
\end{equation}

We now prove item \eqref{I:Cq_row_spans_distinct} by contradiction. Let $p,q$ be two distinct positive numbers.
Assume $[C^p] = [C^q]$, then their Plucker coordinates must be proportionate, for some factor $\lambda$:
\begin{equation}
\label{E:Dij_p_is_proportional_to_Dij_q}
    \D{i,j}^p = \lambda \D{i,j}^q,
\end{equation}
for all $(i,j) \in \binom{[n]}{2}$.
Since the $q$-transform doesn't change $\D{i,j}$ on the odd-numbered orbits, $\lambda = 1$. On the other hand,
for any index $(i,i+k)$ on any even-numbered orbit $O_k$,
\begin{align*}
    \D{i,i+k}^p = p^{2 (-1)^i} \D{i,i+k} \neq q^{2 (-1)^i} \D{i,i+k} = \D{i,i+k}^q,  && \text{since $p \neq q$, }
\end{align*}
which contradicts the proportionality condition \eqref{E:Dij_p_is_proportional_to_Dij_q}.

We now prove item \eqref{I:ECq_equals_EC}. For the cyclic matrix $C$, by Proposition \ref{P:properties_of_C},
\begin{equation}
\label{E:D_ij_of_C}
\begin{aligned}
    \D{i,j} = s_1, & \text{ where $(i,j) \in O_1$,} \\
    \D{i,j} = s_d, & \text{ where $(i,j) \in O_d$,} \\
    s_1 < \D{i,j} < s_d, & \text{ where $(i,j) \in O_k$, and $k \in [2,d-1]$.}
\end{aligned}
\end{equation}
Pass from $\D{i,j}$ to $\D{i,j}^q$ in each of the equations \eqref{E:D_ij_of_C}:
\begin{equation}
\label{E:D_ij_of_Cq}
\begin{aligned}
    \D{i,j}^q = s_1, & \text{ where $(i,j) \in O_1$, since $O_1$ is an odd-numbered orbit,} \\
    \D{i,j}^q = s_d, & \text{ where $(i,j) \in O_d$, since $O_d$ is an odd-numbered orbit, for $n \mod 4 = 2$,} \\
    s_1 < \D{i,j}^q < s_d, & \text{ where $(i,j) \in O_k$, and $k \in [2,d-1]$,}
\end{aligned}
\end{equation}
for all $q$ in a sufficiently small neighborhood of $1$, since the $q$-transform is continuous.
The equations \eqref{E:D_ij_of_Cq} imply
\begin{equation*}
\begin{aligned}
    & \min \D{i,j}^q = s_1, \\
    & \max \D{i,j}^q = s_d, \\
    & E(C^q) = \frac{s_d}{s_1} = E(C),
\end{aligned}
\end{equation*}
for all $q$ in the neighbourhood of $1$.
\end{proof}
\end{proposition}



\end{document}